\documentclass{article}

\usepackage{booktabs}       
\usepackage{amsfonts}       
\usepackage{nicefrac}       
\usepackage{microtype}      

\usepackage{fncylab,enumerate,wrapfig,placeins}
\usepackage{bbm}
\usepackage{graphicx}
\usepackage{enumitem}  
\usepackage{amsmath,amsthm,amsfonts,amssymb,xcolor}  
\usepackage{pifont}

\usepackage{geometry}
\usepackage{leftindex}

\pdfoutput=1  

\usepackage{hyperref}       
\usepackage{url}            

\usepackage{graphicx}  
\usepackage{relsize} 
\usepackage{accents}  

\usepackage{caption}
\usepackage{textgreek,upgreek,bm}

\usepackage{titlesec}

\newtheorem{theorem}{Theorem}[section]
\newtheorem{lemma}[theorem]{Lemma}

\newtheorem{corollary}[theorem]{Corollary}

\newlength{\noteWidth}
\long\def\notes#1{\ifinner
	{\footnotesize #1}
	\else 
	\marginpar{\parbox[t]{\noteWidth}{\raggedright\tiny#1}}  
	\fi\typeout{#1}}
	
\usepackage{cleveref}

\Crefname{corollary}{Corollary}{Corollaries}
\Crefname{eqnarray}{eq.}{eqs.}
\Crefname{equation}{eq.}{eqs.}

\Crefname{figure}{Fig.}{Figs.}
\Crefname{tabular}{Tab.}{Tabs.}
\Crefname{table}{Tab.}{Tabs.}
\Crefname{proposition}{Prop.}{Propositions}
\Crefname{theorem}{Thm.}{Thms.}
\Crefname{definition}{Def.}{Defs.} 
\Crefname{section}{Section}{Sections}
\Crefname{lemma}{Lemma}{Lemmas}
\Crefname{assumption}{Assumption}{Assumptions}

\def\comply{\varphi}

\def\haObj{\widehat{\Obj}}

\def\urls#1{{\footnotesize\url{#1}}}

\def\whamb{\wham{$\bullet$}}

\def\bdd#1{b^{\text{\rm\tiny\ref{#1}}}}

\def\clD{\mathcal{D}}

\def\mindex#1{\index{#1}}

\DeclareFontFamily{U}{mathx}{\hyphenchar\font45}
\DeclareFontShape{U}{mathx}{m}{n}{<-> mathx10}{}
\DeclareSymbolFont{mathx}{U}{mathx}{m}{n}
\DeclareMathAccent{\widebar}{0}{mathx}{"73}

\def\Obj{\Upgamma}  

\newcommand{\bbblot}{\raise1pt\hbox{\vrule height .4ex width .4ex depth .05ex}}

\long\def\defbox#1{\framebox[.9\hsize][c]{\parbox{.85\hsize}{%
\parindent=0pt
\baselineskip=12pt plus .1pt      
\parskip=6pt plus 1.5pt minus 1pt 
 #1}}}

\long\def\beginbox#1\endbox{\subsection*{}%
\hbox{\hspace{.05\hsize}\defbox{\medskip#1\bigskip}}%
\subsection*{}}

\def\endbox{}

 \def\archival#1{} 

\def\FRAC#1#2#3{\genfrac{}{}{}{#1}{#2}{#3}}

\def\ddtp{{\mathchoice{\FRAC{1}{d^{\hbox to 2pt{\rm\tiny +\hss}}}{dt}}%
{\FRAC{1}{d^{\hbox to 2pt{\rm\tiny +\hss}}}{dt}}%
{\FRAC{3}{d^{\hbox to 2pt{\rm\tiny +\hss}}}{dt}}%
{\FRAC{3}{d^{\hbox to 2pt{\rm\tiny +\hss}}}{dt}}}}

\def\ddyp{{\mathchoice{\FRAC{1}{d^{\hbox to 2pt{\rm\tiny +\hss}}}{dy}}%
{\FRAC{1}{d^{\hbox to 2pt{\rm\tiny +\hss}}}{dy}}%
{\FRAC{3}{d^{\hbox to 2pt{\rm\tiny +\hss}}}{dy}}%
{\FRAC{3}{d^{\hbox to 2pt{\rm\tiny +\hss}}}{dy}}}}

\def\limsup{\mathop{\rm lim{\,}sup}}

\def\argmin{\mathop{\rm arg{\,}min}}

\def\state{{\sf X}}

\def\bfr{{\bf r}}

\def\bfmath#1{{\mathchoice{\mbox{\boldmath$#1$}}%
{\mbox{\boldmath$#1$}}%
{\mbox{\boldmath$\scriptstyle#1$}}%
{\mbox{\boldmath$\scriptscriptstyle#1$}}}}

\def\bfmY{\bfmath{Y}}

\def\bfmhhaY{\bfmath{\hhaY}} 
\def\bfmhhaY{\hbox to 0pt{$\widehat{\bfmY}$\hss}\widehat{\phantom{\raise 1.25pt\hbox{$\bfmY$}}}}

\def\hay{{\hat y}}

\def\tilu{\tilde u}

\def\tilf{\tilde f}
\def\tilg{\tilde g}

\def\clA{{\cal A}}

\def\clE{{\cal E}}
\def\clF{{\cal F}}
\def\clG{{\cal G}}
\def\clH{{\cal H}}

\def\clM{{\cal M}}

\def\clU{{\cal U}}

\def\eqdef{\mathbin{:=}}

\def\Expect{{\sf E}}

\def\Proj{\hbox{\sf Proj}}

 \def\epsy{\varepsilon}

\def\varble{\,\cdot\,}

\def\formtmp#1#2{{\vskip12pt\noindent\fboxsep=0pt\colorbox{#1}{\vbox{\vskip3pt\hbox to \textwidth{\hskip3pt\vbox{\raggedright\noindent\textbf{#2\vphantom{Qy}}}\hfill}\vspace*{3pt}}}\par\vskip2pt%
\noindent\kern0pt}}

\def\barb{{\overline {b}}}

\def\barA{{\bar{A}}}

\def\barP{{\bar{P}}}

\def\barmu{{\overline{\mu}}}

\def\bargamma{{\bar{\gamma}}}

{\end{list}}

\def\ass(#1:#2){(#1\ref{#1:#2})}

\def\ritem#1{
\item[{\sf \ass(\current_model:#1)}]
}

\newenvironment{recall-ass}[1]{%
\begin{description}
\def\current_model{#1}}{
\end{description}
}

\def\sq{\hbox{\rlap{$\sqcap$}$\sqcup$}}
\def\qed{\ifmmode\sq\else{\unskip\nobreak\hfil
\penalty50\hskip1em\null\nobreak\hfil\sq
\parfillskip=0pt\finalhyphendemerits=0\endgraf}\fi}

\newcommand{\blot}{\vrule height 1.1ex width .9ex depth -.1ex }
\def\qedb{\ifmmode\blot\else{\vspace{-.2cm}\unskip\nobreak\hfil
\penalty50\hskip1em\null\nobreak\hfil\blot
\parfillskip=0pt\finalhyphendemerits=0\endgraf}\fi}

\newcounter{rmnum}

\newcounter{anum}

\newcommand{\field}[1]{\mathbb{#1}}

\def\Re{\field{R}}

\def\Expect{{\sf E}}

\def\transpose{{\intercal}}

\def\diag{\hbox{\rm diag\thinspace}}

\def\argmin{\mathop{\rm arg\, min}}

\def\epsy{\varepsilon}
\def\varble{\,\cdot\,}

\def\haY{\widehat{Y}}

\def\hhaY{\hbox to 0pt{$\haY$\hss}\widehat{\phantom{\raise 1.25pt\hbox{Y}}}}

\def\hax{\hat x}

\def\haY{\widehat Y}

\newlength{\dhatheight}

\makeatletter
\newcommand\gobblepars{%
	\@ifnextchar\par%
	{\expandafter\gobblepars\@gobble}%
	{}}
\makeatother

\def\whamit#1{\smallbreak\pagebreak[3]%
	\noindent\textit{#1}\ \ \gobblepars}

\def\wham#1{\smallbreak\pagebreak[3]%
	\noindent\textbf{#1}\ \ \gobblepars}
	
\def\whamrm#1{\smallbreak\pagebreak[3]%
	\noindent{{\upshape\rm#1}}\ \ \gobblepars}

\title{Stochastic Online Feedback Optimization
	\\ 
    for Networks of Non-Compliant Agents}

\author{Caio Kalil Lauand\thanks{Caio Kalil Lauand  is with the Division of Systems Engineering, Boston University, Boston, MA, USA.
		Email: {\tt\small cklauand@bu.edu 
		}
 }
 \and
 Andrey Bernstein\thanks{Andrey Bernstein is with the National Renewable Energy Laboratory, Golden, CO, USA.
 Email: {\tt\small andrey.bernstein@nrel.gov}
 }
}

\begin{document}

\maketitle

\begin{abstract}
        In several applications of online optimization to networked systems such as power grids and robotic networks, information about the system model and its disturbances is not generally available. Within the optimization community, increasing interest has been devoted to the framework of online feedback optimization (OFO), which aims to address these challenges by leveraging real-time input-output measurements to empower online optimization. We extend the OFO framework to a stochastic setting, allowing the subsystems  comprising the network (the \emph{agents}) to be \emph{non-compliant}. This means that the actual control input implemented by the agents is a random variable depending upon the control setpoint generated by the OFO algorithm. Mean-square error bounds are obtained for the general algorithm and the theory is illustrated in application to power systems.
\end{abstract}

\clearpage
\tableofcontents

	\clearpage
	\section{Introduction}
	\label{s:Intro}
	This paper aims to leverage stochastic gradient descent (SGD) algorithms to address optimization problems that typically arise in applications to networked systems such as communication systems, power grids and robotic networks \cite{berdal19,dalsim16,bul18}.
	
	In general, such systems are comprised of several interconnected entities (or agents) that are either controlled by a single central controller or multiple local controllers.
	The goal is for these controllers to adjust the system's inputs so that its outputs are steered to an operating point that minimizes a possibly time-varying performance criterion.

	Throughout this paper, the behavior of a networked system of $\clA$ agents with a discretized timescale is modeled as
		\begin{equation}
			y_n=h^{(n)}(x_n) \eqdef h(x_n,r_n) 
			\label{e:sys_map}
		\end{equation}
		in which $n$ is the time index,  $\{ y_n \} \subseteq \Re^m$ represent the system's output variables, $\{ x_n \} \subseteq \Re^d$ its input and $\{r_n\}\subseteq \Re^d$ an exogenous disturbance process. The superscript on $h$ indicates that the system map is time varying; in this case, the time dependency is inherited from $\bfr := \{r_n\}$. 
		
		We consider the setting in which the agents are not necessarily compliant with their controllers, in the sense that given a desired input sequence $\{u_n\} \subseteq \Re^d$, the actual input sequence implemented by the agents is a random variable given by 
		\[
			x_n = \comply(u_n,\Phi_{n+1})
		\]
		where $\comply$ is a function that maps the desired input $u$ to the actual input $x$ and $\{\Phi_n\} \subseteq \state$ an independent and identically distributed (i.i.d.) sequence of random variables. Most of this paper is focused on the linear (or linearized) setting, wherein both the input-output mapping $h$ and the compliance mapping $\comply$ are linear (or affine). In particular,
        \[
        \begin{aligned}
			x_n &= \comply(u_n,\Phi_{n+1}) \eqdef 
			A_{n+1} u_n + b_{n+1}
            \\
           y_n &=  h(x_n,r_n) \eqdef  Cx_n + Dr_n
		\end{aligned}
        \]
		where $A_{n+1} = A(\Phi_{n+1}) \in \Re^{d\times d}$, $b_{n+1} = b(\Phi_{n+1}) \in \Re^{d}$ and $C,D \in \Re^{m\times d} $. A simple example is when $A_{n+1} = \diag (\Phi_{n+1})$ and $b_{n+1} = 0$, with $\Phi_{n+1} \in [0, 1]^d$. This represents the case when the agents choose to implement any control input between $0$ and the desired setpoint $u_n$, with probability determined by the distribution of $\Phi_{n+1}$.

		The constrained optimization problem that is the object of study in this paper takes the form
		\begin{equation}
			u^*_n \in \argmin_{u \in  \clU^{(n)}} 
			\Expect[f^{(n)}(u,\Phi_{n+1})]
			\label{e:opt_prob}
		\end{equation}
		where the above expectation is taken with respect to the distribution of $ \Phi_{n+1}$. The set $\clU^{(n)}$ represents the system's input constraints (e.g., physical or engineering constraints) at time $n$, while $f^{(n)}$ represents the system's performance objectives. In this paper, we quantify the system performance via the additive model
		\[
		f^{(n)}(u,\Phi) = f_x^{(n)}(u,\Phi) + f^{(n)}_y(u,\Phi)
		\]
		with $f_x$ and $f_y$  quantifying the performance in terms of the actual input and output of the system, respectively. Specifically, $f_x$ is a composition of functions $f^{(n)}_x(u,\Phi) = (g^{(n)}_x \circ \comply)(u,\Phi)$, with $g^{(n)}_x (x)$ measuring the performance in terms of the system actual input; whereas  $f^{(n)}_y(u,\Phi) = (g^{(n)}_y \circ h^{(n)} \circ \comply)(u,\Phi)$, with $g^{(n)}_y (y)$ measuring the performance in terms of the system output. Throughout the paper, we consider the convex optimization case, where $\clU^{(n)}$ is a convex set and $f^{(n)}(u,\Phi)$ is a convex function of $u$.

    In typical engineering applications, obtaining solutions or critical points of \eqref{e:opt_prob} in real time might be infeasible due to several challenges. In general, one has no knowledge of the distributions of $\{\Phi_n\}$, the disturbance process $\bfr$ or even the full system model \eqref{e:sys_map}. Moreover, even if all of these pieces of information were available, computing the solution to \eqref{e:opt_prob} in real-time would be computationally expensive.
	
	Past research has focused on addressing some of the above challenges through the framework of online feedback optimization (OFO) \cite{bercomchewan23,berdal19,cheberdevmey20,nonmul21}. Within this framework, measurements of the system's inputs and outputs  are used to drive optimization algorithms in real-time in order to approximately solve \eqref{e:opt_prob}. However, these previous works only considered compliant networks, in which $A(\Phi_{n}) = I$ and $b(\Phi_{n}) =0$ for each $n$. In this special case, \eqref{e:opt_prob} reduces to deterministic optimization:
	$
	\min_{u \in  \clU^{(n)}} f^{(n)}(u)
	$. A prototypical OFO approach to solving this problem is the approximate projected gradient scheme:
	\begin{equation}
		u_{n+1} = \Proj_{\clU^{(n)}}
		\{ u_n 
		- \alpha \nabla\haObj^{(n)}  \}
		\label{e:PSGD_noreg_online}
	\end{equation}
	where $\alpha>0$ is a constant step-size parameter and $\{\nabla\haObj^{(n)}\}$ is a sequence constructed from output measurements $\{\hay_n \}$ with the goal of approximating $\{ \nabla_u f^{(n)}(u_n)\}$. 
	
	The extensive number of success stories of algorithms based upon stochastic approximation (SA) such as stochastic gradient descent (SGD) motivates for a stochastic extension of \eqref{e:PSGD_noreg_online} to address non-compliant networks.
	In this case, the sequence $\{\nabla\haObj^{(n)}\}$ is constructed from input-output measurements $\{\hax_n,\hay_n \}$ and aims to approximate $\{ \nabla_u f^{(n)}(u_n,\Phi_{n+1})\}$.

	\wham{Contributions:}
	We extend the OFO framework to a stochastic setting, addressing situations in which the agents of the networked system of interest are non-compliant. For the general algorithm \eqref{e:PSGD_noreg_online}, we obtain the following upper bound on the mean-squared error (MSE) of estimates:
	\begin{equation}
		\limsup_{N \to \infty} \Expect[\| u_N - u^*_N \|^2]
		\leq 
		\alpha  \upepsilon_a + \upepsilon_b + 
    \frac{1}{\alpha} \Big[\upepsilon_c + \bdd{t:online_tracking_result_limsup} \frac{\upepsilon_c}{\alpha}   \Big] + \frac{\bdd{t:online_tracking_result_limsup} }{\alpha^{3/2}}  \sqrt{\epsilon_c[ \alpha \upepsilon_a + \upepsilon_b]}
		\label{e:MSE_lim}
	\end{equation}
	in which $\bdd{t:online_tracking_result_limsup}$ is a constant, $\upepsilon_a$ depends upon the volatility of $\{A_{n+1}\}$, $\upepsilon_b$ is related to the error in approximating $\{ \nabla_u f^{(n)}(u_n,\Phi_{n+1})\}$ through input-output measurements and $\upepsilon_c$ depends on the time-variability of the optimization problem \eqref{e:opt_prob}.


	
	In the setting of static optimization with full gradient information (i.e., $\upepsilon_b = \upepsilon_c =0$), the bound \eqref{e:MSE_lim} coincides with the MSE bound expected for a SGD algorithm with constant step-size \cite{bor20a,bacmou11}. If in addition the sequence $\{A_{n+1}\}$ is time-invariant (i.e., $\Phi_0 \equiv \Phi_n$ for all $n$), it follows that $\upepsilon_a = 0$, leading to convergence as expected from the deterministic OFO algorithm in \cite{berdal19}. 
	
	\wham{Relevant Literature:} 
\whamit{Online feedback optimization (OFO). } Interest in the OFO framework has been growing within the controls community, particularly motivated by applications to power systems \cite{berdal19,bercomchewan23,cheberdevmey20}. We point the reader to \cite{bercomchewan23} for a more comprehensive survey of OFO. 

The papers \cite{berdal19,cheberdevmey20} tackle the compliant case with additional inequality constraints. While \cite{berdal19} requires partial knowledge about the system map, \cite{cheberdevmey20} eliminates this limitation by leveraging gradient-free optimization methods. Another approach to overcome this challenge may be found in \cite{nonmul21}. Similarly to the setting of the present paper, many of the references in OFO are restricted to convex optimization; see \cite{habhauortboldor20} for an extension to non-convex problems.

 The reference \cite{woobiadal23}, which is perhaps most closely related to our work, studies an application of primal-dual methods to static optimization without using input-output measurements. They
 allow for the distribution of $ \Phi_{n+1}$ to depend upon $u_n$ and obtain moment bounds independent of  $\alpha$ between $\{u_n\}$ and a performatively stable point $u^{*p}$. This generally differs from the desired optimal point $u^*$, but is optimal for the distribution that $u_n$ induces on $ \Phi_{n+1}$. Bounds on the error $\| u^{*p} - u^* \|$ depend upon the problem's parameters (e.g., Lipschitz constants and strong convexity parameters) but are also not related to the step-size $\alpha$. 
 


\whamit{Stochastic Approximation. } Many machine learning and optimization algorithms are built upon stochastic approximation. This framework was born in the 1950s in the seminal work \cite{robmon51}, and theoretical refinements have appeared ever since \cite{bacmou11,bor20a,chezhadoaclamag22,laumey24}. Most relevant to the present work is the paper \cite{bacmou11} which obtains finite-time MSE bounds  of order $O(\alpha)$ for SGD. Although the assumptions of this prior work are similar to the ones in the present paper, it concerned a static unconstrained optimization problem and assumed that $\{ \nabla_u  f^{(n)}(u_n,\Phi_{n+1}) \}$ could be measured directly.


	\wham{Organization:} The remainder of the paper is organized in three additional sections. \Cref{s:main} sets the stage for analysis by going over the main assumptions and notation imposed in the paper as well as presenting the main results. These results are then illustrated through simple numerical experiments in \Cref{s:exp}. Moreover, \Cref{s:exp} contains an application of the stochastic OFO algorithm in the context of power systems. Conclusions and directions for future research are contained in \Cref{s:conc}, while technical proofs can be found in the Appendix.

	\section{Main Results}
	\label{s:main}
	\subsection{Preliminaries}
	\label{s:assump}
	\wham{Notation:} We use
	$\| \varble \|$ to denote the Euclidean norm for vectors and the induced operator norm for matrices. For a random variable $X$, which may be vector or matrix-valued, we write
	\[
	\Expect[ \| X\|^p] = \| X\|_p^p 
	\, , 
	\qquad 
	\Expect[ \| X\|^p \mid \clF_n] = \| X\|_{p,n}^p 
	\]
	to denote its $L_p$ moments and conditional moments with respect to the filtration of the history generated by \eqref{e:PSGD_noreg_online} up to time $n$: $\clF_n \eqdef \{ u_0, A_1 , b_1, A_2 , b_2\cdots , A_n,b_n\}$. Moreover, we denote $\barA \eqdef \Expect[A(\Phi_n)]$ and $\barb \eqdef \Expect[b(\Phi_n)]$.

	For each $n$, we use the following short-hand notation for tracking errors:  \, $\tilu_n \eqdef u_n - u^*_n$,
	\[
	\nabla_u  \tilf^{(n)}(u_n,\Phi_{n+1}) \eqdef \nabla_u f^{(n)}(u_n,\Phi_{n+1}) - \nabla_u f^{(n)}(u^*_n,\Phi_{n+1})
	\]



	\wham{Assumptions:} 

	\wham{(A1)} The functions $g^{(n)}_x,g^{(n)}_y$ are  strongly convex with parameter $\mu$ and continuously differentiable.  Moreover, their gradients are Lipschitz continuous: there is $L_g<\infty$ such that for each $n$,
	\[
	\begin{aligned}
		\| \nabla g^{(n)}_y(y) - \nabla  g^{(n)}_y(y')   \|
		&\leq
		L_g \| y-y'   \|
		\\
		\| \nabla g^{(n)}_x(x) - \nabla  g^{(n)}_x(x')   \|
		&\leq
		L_g \| x-x'   \|
	\end{aligned}
	\]
	for all $y,y' \in \Re^m$ and  $x,x' \in \Re^d$.
	
	\wham{(A2)} For each $n$, the set $\clU^{(n)} \subset \Re^d$ is convex and compact. Moreover, the sequence $\{\clU^{(n)} \}$ is uniformly bounded: $b_{\clU} \eqdef \sup_{n\geq 1} \sup_{u \in \clU^{(n)}} \| u \| < \infty$.
	
	\wham{(A3)} The sequence $\{ \Phi_n\}$ is i.i.d. and there exists a constant $\sigma_\Delta<\infty$ such that for all $n$,
	\[
	\Expect[ \|A_{n+1} \|^4 \mid \clF_n ]  \leq \sigma^4_\Delta
	\, , 
	\quad 
	\Expect[ \|b_{n+1}\|^4 \mid \clF_n ] \leq \sigma^4_\Delta
	\]
	Moreover, $\sup_n\| r_n\| \leq \sigma_\Delta$.
	
	\wham{(A4)} The sequences of measurements $\{\hay_n , \hax_n \}$ admit the bounds: for a constant $\epsy_m$ and all $n$,
	\[
	\Expect[\| \hay_n - y_n   \|^4 \mid \clF_n ]
	\leq \epsy^4_m
	\, , \quad 
	\Expect[\| \hax_n - x_n   \|^4 \mid \clF_n ]
	\leq \epsy^4_m
	\]
	Moreover, it is assumed that the sequence $\{A_{n+1}\}$ can be recovered from the observations $\{ \hax_n\}$ with precision $\epsy_m$: there exists a matrix-valued sequence $\{A^\circ_{n+1}\}$ constructed from $\{ \hax_n, u_n\}$ such that for all $n$,
	\[
	\Expect[\| A^\circ_{n+1} - A_{n+1}  \|^4 \mid \clF_n ]
	\leq 
	\epsy^4_m
	\]
	
	\wham{(A5)}  The matrices $C \barA$ and $\barA$ are of full column rank. 

	Assumptions (A1) and (A2) are common in OFO \cite{berdal19,bercomchewan23}, while (A3) ensures that $\{ \nabla_u f^{(n)}(u_n,\Phi_{n+1}) - \Expect[\nabla_u f^{(n)}(u_n,\Phi_{n+1})]\}$ is a martingale difference sequence, as assumed in a vast part of the SA/SGD literature \cite{bor20a,bacmou11}.
	
	The conditions in (A4) are a slight strengthening of the measurement error conditions in the OFO literature: while a $L_2$ bound is typically assumed, we require a $L_4$ bound due to the presence of multiplicative noise in the algorithm \cite{berdal19,cheberdevmey20}. One special case in which $\{A_{n+1}\}$ can be recovered exactly from $\{ u_n, \hax_n\}$ is when $A_{n+1} = \diag(\Phi_{n+1})$, $b_{n+1} =0$ and $x_n = \hax_n$, leading to $\{\hax^i_n\}$ independent of $\{u^j_n: j \neq i \}$.
    For each $n$, $\{A_{n+1}\}$ is obtained via $A^{i,i}_{n+1}  = {A^\circ}^{i,i}_{n+1} = {\hax}^i_n/u^i_{n}$ for $1 \leq i \leq d$. This choice is employed in the experiments surveyed in \Cref{s:exp}.

	Assumption (A5)  is imposed so that $\nabla_u f$ is strongly monotone in its first variable in conditional mean. In applications where this assumption is not satisfied, one could implement a regularized version of the algorithm (see the discussion following \Cref{t:online_tracking_result_limsup}).
	
	\subsection{Feedback-Based SGD and Tracking with Partial Information}
    %
	In the scenario in which exact measurements of $\{ \nabla_u f^{(n)}(u_n,\Phi_{n+1})\}$ are available at each $n$, a projected stochastic gradient descent algorithm can be formulated as follows:
	\[
	\begin{aligned}
		&u_{n+1} =   \Proj_{\clU^{(n)}}
		\{ u_n 
		- \alpha \nabla_u f^{(n)}(u_n,\Phi_{n+1}) \}
		\\
		&\nabla_u f^{(n)}(u_n,\Phi_{n+1}) = {A_{n+1}}^\transpose [ C^\transpose \nabla g_y^{(n)}(y_n) + \nabla g_x^{(n)}(x_n)]
	\end{aligned}
	\]
	
	Under the assumptions of this paper, the above choice satisfies assumptions that have been previously imposed within the SGD literature \cite{bacmou11}, and we have the following lemma.
	\begin{lemma}
		\label[lemma]{t:facts_f}
		Suppose that (A1)--(A3) hold.
		\whamrm{(i)}  Then, $\nabla_u f^{(n)}$ is Lipschitz continuous in quadratic mean in its first variable: there is $L_f<\infty$ such that for each $n$,
		\[
		\Expect[\| 
		\nabla_u \tilf^{(n)}(u_n,\Phi_{n+1})   \|^2
		\mid \clF_n] 
		\leq 
		L_f^2 \| \tilu_n   \|^2
		\]
		
		\whamrm{(ii)} If in addition (A5) holds, $\nabla_u f^{(n)}$ is strongly monotone in its first variable in conditional mean: there is a constant $\barmu_f>0$ such that for each $n$,
		\[
		\Expect[	\nabla_u 
		f^{(n)}(u_n,\Phi_{n+1}) 
		^\transpose 
		(u_n - u^*_n) \mid \clF_n] 
		\geq 
		\barmu_f \| \tilu_n   \|^2
		\]
		
		\whamrm{(iii)} There exists a constant $\sigma_f$ such that the following holds for each $n$:
		\[
		\Expect[\|  \nabla_u f^{(n)}(u_n^*,\Phi_{n+1}) -
        \Expect[ \nabla_u f^{(n)}(u_n^*,\Phi_{n+1})]
        \|^2 \mid \clF_n]
		\leq 
		\sigma^2_{f}
		\]
		
	\end{lemma}
	The proof of \Cref{t:facts_f} is deferred to the Appendix.

	In the setting of this paper, however, precise information about the system model, its disturbances, and its randomness is not available; hence, exact measurements of $\{\nabla_u  f^{(n)}(u_n,\Phi_{n+1})\}$ cannot be obtained. Instead, we apply the approximate feedback-based SGD algorithm \eqref{e:PSGD_noreg_online}, in which 
	\begin{equation}
		\nabla \haObj^{(n)}= {A_{n+1}^\circ}^\transpose C^\transpose \nabla g_y^{(n)}(\hay_n) +{A_{n+1}^\circ}^\transpose\nabla g_x^{(n)}(\hax_n)
		\label{e:chainrule_online}
	\end{equation}

    Similarly to \cite{berdal19}, we required exact knowledge of the Jacobian of $h$ for each $n$. Other implementations may approximate or estimate this quantity by linearization of the system map around an operating point or application of machine learning and gradient-free optimization techniques \cite{berdal17,nonmul21,cheberdevmey20}.
	
	\begin{theorem}
		\label[theorem]{t:tracking_result_online}
		Suppose that (A1)--(A5) hold. Suppose in addition that the update rule \eqref{e:PSGD_noreg_online} is implemented and that $\nabla \haObj^{(n)}$ is of the form \eqref{e:chainrule_online} for each $n$.
		Then, the mean squared tracking error admits the bound:
		\begin{equation}
			\| \tilu_{N}   \|_2^2
			\leq 
			\Upsilon_\alpha^{N} \| \tilu_0    \|_2^2
			+  \sum_{k=0}^{N-1} \Upsilon_\alpha^{k} [\psi^2_{N-k-1} 
			+  q_\alpha + 2\beta_{N-k-1} ]
			\label{e:track_finitetime_online}
		\end{equation}
		where $\Upsilon_\alpha = 1-2\alpha \barmu_f + \alpha^2 L_{f}^2$, $\psi_{n} = \| u^*_{n+1} -  u^*_{n}  \|$, 
		\[
		\begin{aligned}
		\beta_{n-1} &= \psi_{n-1} \Big(\sum_{i=0}^{n-1} \Upupsilon^{i/2}_\alpha \sqrt{q_\alpha}
        +
        \sqrt{\Upupsilon_\alpha} \sum_{i=0}^{n-2} \Upupsilon_\alpha^{i/2} \psi_{n-i-2}
		+  \Upupsilon^{n/2}_\alpha \Expect[\| \tilu_0 \|] \Big) \, , 
		\\
		q_\alpha &= \bdd{t:one_step_couple_haobj}  [ \alpha^2 (2 \sigma_f^2  + \sigma_f+ 2 \epsy^2_m + \epsy_m) + \alpha\epsy_m ] \, , 
		\end{aligned}
		\]
		and
        $\bdd{t:one_step_couple_haobj}$ is a constant depending upon $b_\clU$ and $L_f$.
	\end{theorem}

\begin{figure*}
		\centering
		\includegraphics[width = \textwidth]{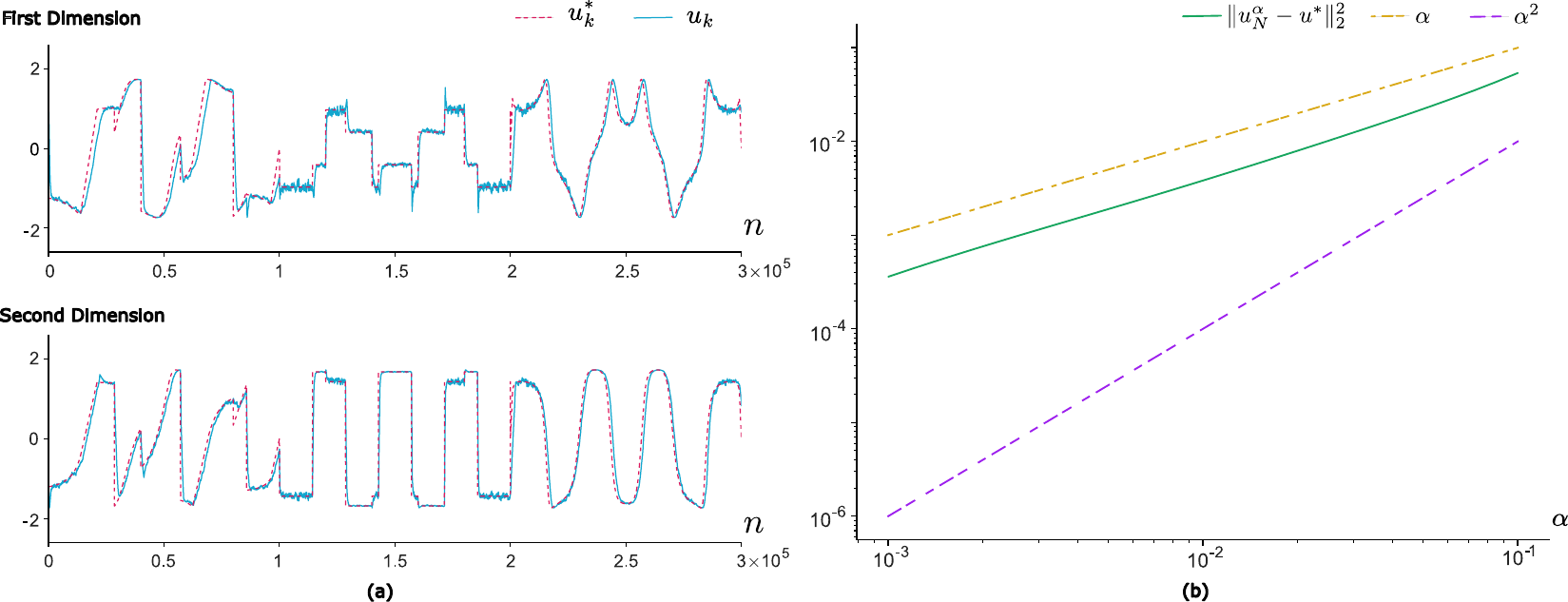}
		\caption{(a) Tracking of a moving optimizer; (b) Steady state MSE for a static optimization problem.}
		\label{fig:Tracking_perf}
	\end{figure*}

	In view of the finite-time bound in \Cref{t:tracking_result_online}, the  bound in \eqref{e:MSE_lim} is obtained under additional conditions on the step-size and the time-variability of the problem.

	\begin{theorem}
		\label[theorem]{t:online_tracking_result_limsup}
		Suppose the assumptions of \Cref{t:tracking_result_online} hold. Suppose, in addition, that there exists $\bargamma<\infty$ such that $\psi_n\leq \bargamma$ for all $n$, and that the step-size satisfies $\alpha < \frac{\barmu_f}{ L^2_f}$. Then, the MSE bound \eqref{e:MSE_lim} holds with
		\[ \begin{aligned}
		\upepsilon_a & \eqdef \frac{1}{\barmu_f} \bdd{t:one_step_couple_haobj}[2 \sigma^2_f + \sigma_f]  \, , \qquad \quad \, \,\qquad  \upepsilon_c \eqdef \frac{1}{\barmu_f} \bargamma^2
		\\
        \upepsilon_b &\eqdef \frac{1}{\barmu_f} \bdd{t:one_step_couple_haobj} (\epsy_m + \alpha [2\epsy^2_m + \epsy_m] )
		  \, , \quad  
		\bdd{t:online_tracking_result_limsup} \eqdef \frac{4}{\barmu_f}
		\end{aligned}
		\]
	\end{theorem}
    \smallskip

    The proofs of \Cref{t:tracking_result_online,t:online_tracking_result_limsup} are given in the Appendix.
	
	\wham{Importance of (A5):} While assumption (A5) is necessary for the bounds in \Cref{t:online_tracking_result_limsup} to hold, it is likely that this assumption will not be satisfied in several applications. Instead, one could consider a regularized version of \eqref{e:opt_prob}: for a small constant $\eta>0$,
	\begin{equation}
		u^{r*}_n \in \argmin_{u \in  \clU^{(n)}} 
		\Expect[f^{(n)}(u,\Phi_{n+1}) + \tfrac{1}{2}\eta \| u \|^2]
		\label{e:opt_prob_reg}
	\end{equation}
	
    An analogous bound to the one in  \Cref{t:online_tracking_result_limsup} is obtained  for this case, by solving \eqref{e:opt_prob_reg} via \eqref{e:PSGD_noreg_online} with
	\begin{equation}
    \nabla \haObj^{(n)}= {A_{n+1}^\circ}^\transpose [C^\transpose \nabla g_y^{(n)}(\hay_n) + \nabla g_x^{(n)}(\hax_n)] + \eta u_n 
    \label{e:PSGD_reg}
    \end{equation}
	
	\begin{corollary}
	\label[corollary]{t:online_reg}
	Suppose that (A1)--(A4) hold and that \eqref{e:PSGD_noreg_online} is implemented with $\nabla \haObj^{(n)}$ of the form \eqref{e:PSGD_reg}. Suppose moreover  there exists $\bargamma<\infty$ such that $\psi_n\leq \bargamma$ for all $n$, and that the step-size satisfies $\alpha < \frac{\eta}{ (L_f+\eta)^2}$.
    Then, the following MSE bound holds:
\[
			\limsup_{N \to \infty} \Expect[\| u_N - u^{r*}_N \|^2]
			\leq 
			\alpha  \upepsilon^r_a + \upepsilon^r_b + \frac{1}{\alpha} \Big[\upepsilon^r_c  + \frac{\bdd{t:online_reg} }{\alpha} \upepsilon^r_c  \Big] + \frac{\bdd{t:online_reg} }{\alpha^{3/2}}  \sqrt{\epsilon^r_c[ \alpha \upepsilon^r_a + \upepsilon^r_b]}
\]
in which 
		\[ \begin{aligned}
		\upepsilon^r_a & \eqdef \frac{1}{\eta} b^r [ 2\sigma^2_f + \sigma_f] \, , \qquad \quad \, \,\qquad  \upepsilon^r_c \eqdef \frac{1}{\eta} \bargamma^2 
		\\
        \upepsilon^r_b &\eqdef \frac{1}{\eta} b^r (\epsy_m + \alpha[2\epsy^2_m + \epsy_m ] )
		  \, , \quad  
		\bdd{t:online_reg} \eqdef \frac{4}{\eta}
		\end{aligned}
		\]
and $b^r$ is a constant depending upon $b_\clU$, $L_f$ and $\eta$.
	\hfill 
\end{corollary}
	
	\section{Numerical Experiments}
	\label{s:exp}
	\subsection{Toy Problems}
	The first numerical example investigated in this Section was designed to test the tracking capability of the algorithm as well as the MSE bound in \Cref{t:online_tracking_result_limsup}.
	Consider \eqref{e:opt_prob} with
	\[
	\begin{aligned}
		f^{(n)}_y(u_n,\Phi_{n+1}) &= y^\transpose y \, , \quad f^{(n)}_x(u_n,\Phi_{n+1})=0 \, ,
		\\
		y_n &= C x_n + D  r_n  \, , \quad y_n,r_n \in \Re^2
		\\
		x_n &= A_{n+1} u_n \, ,  \qquad u_n,x_n \in \Re^2
	\end{aligned}
	\]
	For each $n$, $A_{n+1} = \diag(\Phi_{n+1})$ in which $\Phi_{n+1} \in \Re^2$ with entries $\Phi^i_{n+1} \sim \text{Unif}[0,1]$.	The matrix $C$ is of the form $C = \diag(\nu)$ where $\nu \in \Re^2$ has entries $\nu^i \sim \text{Unif}[-5,0]$ and $D \in \Re^{2 \times 2}$ has entries sampled independently from $\text{Unif}[0,1]$. We note that (A5) is satisfied for this model so no regularization is needed.

	In this example, $\{A_n\}$ can be fully recovered from $\{\hax_n,u_n\}$, as explained at the end of \Cref{s:assump}. To mimic real-life output measurements, the following observation model was adopted: $\hay_n = y_n + w^\bullet_{n+1}$, in which  $w^\bullet_n \sim N(0,I) $  for each $n$.

	\wham{Tracking:}
	For a fixed simulation runlength $N = 3 \times 10^{5}$, the sequence $\{r_n\}$ was chosen as
	\[
	r_n = 
	\begin{cases}
		a_n \Lambda(\omega^\circ n) \, ,& \quad n \leq N/3
		\\
		a_n \Pi(\omega^\circ n) \, ,& \quad N/3+1\leq n \leq 2N/3
		\\
		a_n\sin(\omega^\circ n) \, ,& \quad 2N/3+1\leq n \leq N
	\end{cases}
	\]
	where the notation $\Pi, \Delta$ denotes the unit square and triangle waves, respectively. The frequencies and amplitudes were $\omega^\circ = [5,7]^\transpose \times 10^{-4}$ and $a_n = [10 , \zeta_n]^\transpose$, in which
	\[
	\zeta_n= 
	\begin{cases}
		7\, ,& \quad n \leq N/3
		\\
		15\, ,& \quad N/3+1\leq n \leq 2N/3
		\\
		13 \, ,& \quad 2N/3+1\leq n \leq N
	\end{cases}
	\]
	
	The algorithm \eqref{e:PSGD_noreg_online} was applied without regularization and with $\alpha = 2 \times 10^{-3}$. The projection sets were chosen to be time-invariant and of the form $\clU^{(n)} \equiv \clU \{ u: u^\transpose u \leq 9  \}$.
    
    \Cref{fig:Tracking_perf}~(a) shows plots of the sequences $\{ u^i_n , {u^*}^i_n : 1 \leq i \leq 2\}$ as functions of $n$. As expected from \Cref{t:online_tracking_result_limsup}, the sequence $\{u_n\}$ reaches a neighborhood around the moving optimizer $\{u^*_n\}$ after a transient period and is able to track this optimal trajectory with a bounded error.


	\wham{Estimation Error:}
	To test the MSE bound in \Cref{t:online_tracking_result_limsup}, the sequence $\{r_n\}$ was chosen as $r_n \equiv r = [2,1]^\transpose$, so that $\bargamma = 0$. For a fixed simulation runlength of $N=3 \times 10^{5}$, several step-size values in the range  $\alpha \in [10^{-3}, 10^{-1}]$ were tested in independent experiments with common initial condition $u_0 = u^*$. Each experiment used the same noise sequence $\{ A_{n} \}$ and $w^\bullet_n =0$, so that $\epsy_m =0$.
	
	For each fixed $\alpha$, the empirical MSE $\| u^\alpha_{N} - u^*  \|_2^2$ was estimated via Monte Carlo as follows: letting $\{u^\alpha_n\}$ be a sequence of estimates of $u^*$ obtained from \eqref{e:PSGD_noreg_online} with step-size $\alpha$, the MSE was approximated empirically through
	\[
	\| u^\alpha_{N} - u^*  \|_2^2 \approx  \frac{1}{N - N^\circ+1}  \sum_{k = N^\circ }^N \| u^\alpha_{k} - u^*  \|^2 
	\]
	where $N^\circ =\lfloor0.5 N \rfloor $.

    \Cref{fig:Tracking_perf}~(b) shows a plot of $\{ \| u^\alpha_{N} - u^*  \|_2^2 \}$ as a function of $\alpha$ in a logarithmic scale. Also plotted for comparison with the expected bounds are the functions $\tau_1(\alpha) = \alpha$ and $\tau_2(\alpha) = \alpha^2$.
	We see that the empirical MSE scales with $\alpha$, as expected from the bounds in \Cref{t:online_tracking_result_limsup} for a time-invariant optimization problem with no observation noise (i.e., $\epsy_m = \bargamma = 0$).
	
	\subsection{Real-Time Optimal Power Flow}
	The next example aims to illustrate an application of the stochastic OFO framework within the context of power systems optimization. The goal is to optimize the operation of collections of distributed energy resources (DERs) in a power distribution network in real time. Similarly to what was done in \cite{bercomchewan23,berdal19,cheberdevmey20}, we frame this task as a time-varying optimal power flow (OPF) problem.

\begin{wrapfigure}{r}{0.5\textwidth}
  \begin{center}
    \includegraphics[width=0.48\textwidth]{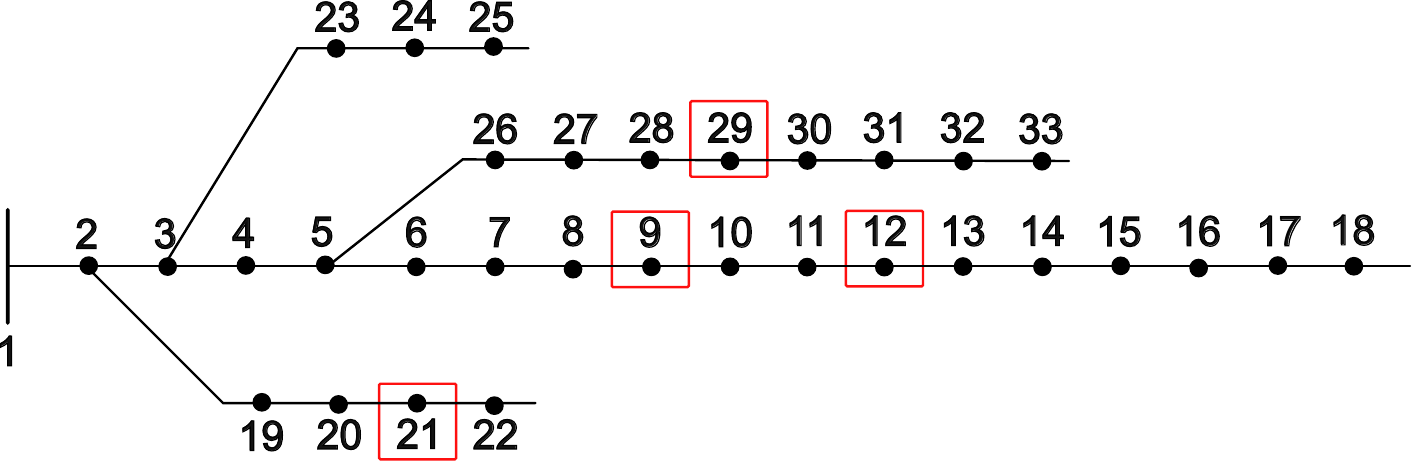}
  \end{center}
		\caption{Schematic for the 33-node test network \cite{meswu89}. Boxed nodes represent the controllable PV units and node $1$ is the feeder head. The remaining nodes are uncontrollable loads.}
		\label{fig:33Bus}
\end{wrapfigure}

	\wham{Setup:} We consider the IEEE 33-node test feeder \cite{meswu89}, in which the controllable nodes (agents) were populated with photovoltaic (PV) systems (see \Cref{fig:33Bus} for details). 
	
	For each $n$, the desired input associated with the $i^{\text{th}}$ agent in the network is $u^i_n = (P^i_n,Q^i_n)^\transpose \in \Re^{2} $, in which $P^i_n$ and $Q^i_n$ represent desired net active and reactive power injections to the distribution network, respectively. The actual net active and reactive power injections applied to the system by the $i^{\text{th}}$ agent is denoted $x^i_n = (P^{x,i}_n,Q^{x,i}_n)^\transpose \in \Re^{2}$ so that $u_n, x_n \in \Re^{2 \clA}$. The compliance model was chosen so that $\{A_n\}$ could be recovered exactly from $\{ \hax_n , u_n\}$ (see the discussion at the end of \Cref{s:assump}): 
	\[
    \begin{aligned}
    	\hax_n &= x_n = 
	\begin{bmatrix} A_{n+1} P_n & Q_n
	\end{bmatrix}^\transpose 
    \\
    A^\circ_{n+1} &=A_{n+1} = \diag(\Phi_{n+1})
    \end{aligned}
	\]
    where $\Phi_{n+1} \in \Re^\clA$.
	The system's outputs are the voltage magnitudes measured at each of the agents $y_n \in \Re^\clA$. 
	
	As mentioned after \eqref{e:chainrule_online}, the OFO architecture proposed in this paper requires knowledge of the Jacobian of the system map. We approximate the mapping from $x_n$ to $y_n$ by a linear relationship based upon \cite{berdal17}: for each $n$,
	\[
	y_n = C_1 P^x_n + C_2 Q^x_n + D r_n
	\]
	in which $\{D r_n\}$ models the voltage contributions from the uncontrollable loads in the network.
	
	While the above linearized model is used to implement the OFO algorithm, we note that the output measurements $\{\hay_n\}$ are obtained by solving the exact nonlinear power flow equations for each $n$.

	The regularized optimization problem \eqref{e:opt_prob_reg} was considered with
	\[\begin{aligned}
	f^{(n)}_x(u_n,\Phi_{n+1})
	&= \sum_{i=1}^\clA \kappa_P (P^{x,i}_n - \barP^i_n)^2 + \kappa_Q (Q^{x,i}_n)^2 
    \\
    f^{(n)}_y(u_n,\Phi_{n+1}) &=  \sum_{i=1}^\clA \kappa_y   (y^i_n - 1)^2
    \end{aligned}
	\]
	where $\kappa_P,\kappa_Q,\kappa_y$ are positive constants and $\{\barP_n^i\}$ represents the real power available at the $i^{\text{th}}$ agent for each time instant $n$. The feasible set for the $i^{\text{th}}$ agent represents the PV inverter constraints:
	\[
	\clU_i^{(n)} = \{ (P,Q): P^2 + Q^2 \leq S^2_{i,\max} \, , 0 \leq P \leq \barP^i_n  \}
	\]
	in which $S_{i,\max}$ is the inverter rating for agent $i$.

\begin{wrapfigure}{r}{0.5\textwidth}
  \begin{center}
    \includegraphics[width=0.48\textwidth]{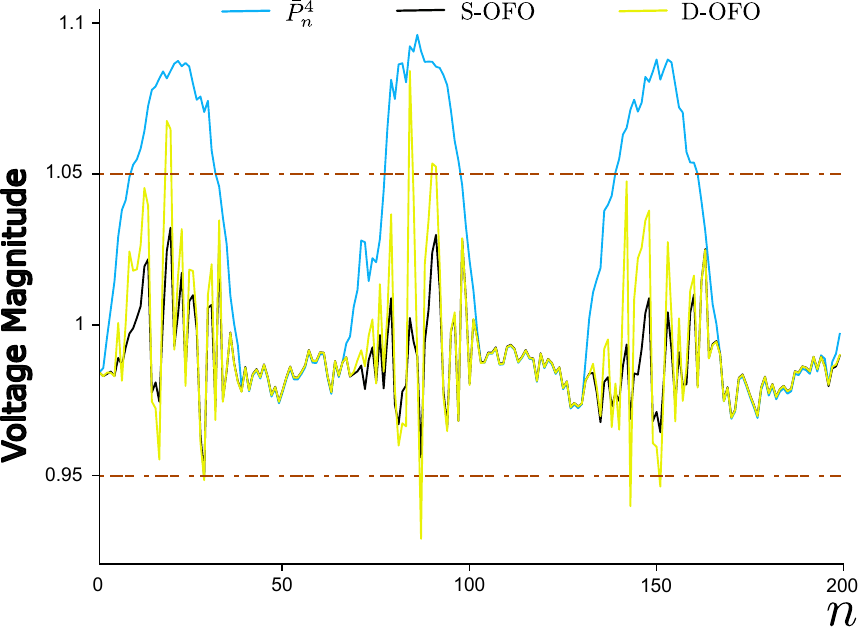}
  \end{center}
		\caption{Voltage magnitude profiles for the $4^{\text{th}}$ agent (node $29$).}
		\label{fig:volt_MAG}
\end{wrapfigure}

\wham{Results:}  
	Load profiles for the uncontrollable nodes of the network consisted of real power usage data obtained from the Smart* UMass apartment dataset. The  sequences representing the  real power available at each time instant at the PV generators of the controllable nodes $\{\barP_n^i: 1 \leq i \leq \clA \}$ were obtained from the Smart* UMass solar panel dataset \cite{barmisirwcecshealb12}.

    Two OFO architectures were implemented for comparison: the stochastic algorithm proposed by this paper \eqref{e:PSGD_reg} (S-OFO) and the deterministic algorithm in \cite{berdal19} (D-OFO), in which $
		\nabla \haObj^{(n)}= C^\transpose \nabla g_y^{(n)}(\hay_n) + \nabla g_x^{(n)}(\hax_n) + \eta u_n$.
 Both choices were implemented with $\alpha = 5 \times 10^{-2}$ and $\eta = 10^{-3}$, which were selected by trial and error. The constants $\kappa_P = 4$, $\kappa_Q = 1$ and $\kappa_y = 8$ were chosen for the objective function.

    The two algorithms were compared over $M=100$ independent experiments with equal load profiles and real available power curves, but different initial conditions and volatility: $\{ \leftindex^j{u}_0, \leftindex^j{\Phi}_n : 1 \leq j \leq M\}$. Performance was based upon two metrics: Approximate expected average power curtailment (PC) and approximate expected average voltage deviation (VD):
    \[
    \begin{aligned}
      \textbf{PC:} \quad  &
      \frac{1}{\clA} \frac{1}{N} \frac{1}{M}  \sum_{i=1}^{\clA}
         \sum_{k=1}^{N} \sum_{j=1}^{M} (\leftindex^j{P}^{x,i}_k - \barP^i_k  )^2
      \\
       \textbf{VD:} \quad & \frac{1}{\clA} \frac{1}{N} \frac{1}{M} \sum_{i=1}^{\clA}
        \sum_{k=1}^{N} 
                 \sum_{j=1}^{M} 
                 (\leftindex^j{\hay}^i_k - 1  )^2
    \end{aligned}
    \]

\Cref{tab:opf} displays results from several experiments with different choices of $\{ \Phi_n\}$. It is possible to see more power curtailing resulting from the D-OFO algorithm when there is more uncertainty at the controllable nodes. That is, when there is a larger presence of unexpected loads contributing to the power injections at these nodes.

Irrespective of the choice of $\{ \Phi_n\}$, however, we consistently see much less voltage violations resulting from S-OFO when compared to its deterministic counterpart. A similar pattern was also observed for a different choice of objective that aimed to penalized power curtailment over voltage deviations, in which $k_P=8$, $k_Q=1$ and $k_y=4$.

This advantage of S-OFO over D-OFO in terms of voltage is further illustrated in \Cref{fig:volt_MAG}, which shows voltage magnitude profiles as functions of $n$ for  the $4^{\text{th}}$ agent (node $29$) when $\Phi_n \sim \text{Unif[-1/2,1]}$.
Also plotted in \Cref{fig:volt_MAG} are lines at $1.05$~p.u. and $0.95$~p.u., which define a desired operating region for the system. While we do not enforce this through inequality constraints in the optimization problem as in \cite{berdal19}, we include it in the plot to enable comparison.

We see from \Cref{fig:volt_MAG} that in the case of zero power curtailing (i.e., $P^{x,4}_n = \barP^4_n$), voltages repeatedly surpass the desired upper bound of $1.05$~p.u.. This is improved  when the D-OFO algorithm is applied, but sporadic voltage violations are still apparent. Voltage magnitudes are more stabilized when the S-OFO algorithm is used.

\begin{table}[h]
\caption{Comparison between the S-OFO and D-OFO algorithms for various choices of $\{\Phi_n\}$.}
\begin{center}
\label{tab:opf}
\setlength\tabcolsep{0pt} 

\begin{tabular*}{\columnwidth}{@{\extracolsep{\fill}} ll cccc}
\toprule
      $\{\Phi_n \}$& Support &
     \multicolumn{2}{c}{S-OFO} & 
     \multicolumn{2}{c}{D-OFO} \\ 
\cmidrule{3-6}
     & & PC $(kW^2)$ & VD $(p.u.)$ & PC $(kW^2)$ & VD $(p.u.)$ \\
\midrule
     Beta$(4,2)$ & $[0,1]$  & $110$ & $4\times 10^{-4}$ & $85.8$ & $4.6\times 10^{-4}$ \\
     Beta$(2,4)$  & $[0,1]$  & $273.7$ & $1.8 \times 10^{-4}$ & $216.6$ & $2 \times 10^{-4}$ \\
     Uniform & $[0,1]$  & $201$ & $2.7 \times 10^{-4}$ & $161.6$ & $3.4 \times 10^{-4}$  \\
\addlinespace
     Beta$(4,2)$ & $[-0.5,1]$  & $194.3$ & $2.7 \times 10^{-4}$ & $155.4$ & $3.4 \times 10^{-4}$ \\
     Beta$(2,4)$  & $[-0.5,1]$  & $435.6$ & $2.4 \times 10^{-4}$ & $460.7$ & $3.1 \times 10^{-4}$ \\
     Uniform & $[-0.5,1]$  & $364.6$ & $2.3 \times 10^{-4}$ & $331.1$ & $3.4 \times 10^{-4}$  \\
\addlinespace
      Beta$(4,2)$  & $[-1,1]$  & $303.3$ & $2.2 \times 10^{-4}$ & $255.1$ & $3 \times 10^{-4}$ \\
      Beta$(2,4)$  & $[-1,1]$  & $436.3$ & $2.4 \times 10^{-4}$ & $811.5$ & $10.7 \times 10^{-4}$ \\
     Uniform & $[-1,1]$   & $442$ & $2.7 \times 10^{-4}$ & $553.9$ & $6.7 \times 10^{-4}$ \\
\bottomrule
\end{tabular*}
\end{center}
\end{table}

	
	

	\section{Conclusions}
	\label{s:conc}
	This paper has extended the OFO framework to a stochastic setting, in which agents may not be compliant to the commands issued by the central/multiple local controllers. There are many paths for further research:
	
	
	\whamb The results in this paper restrict to the special case in which $\{\Phi_n\}$ is an i.i.d. sequence. Could we relax (A3) and allow for the the case in which $\{ \Phi_n\}$ is a Markov chain?
    
	
	\whamb Prior work in SA/SGD has established that passing $\{ u_n\}$ through a low-pass filter leads to better MSE bounds when $\{\Phi_n\}$ is a mixture of sinusoids \cite{laumey25}. Could we apply such techniques to the stochastic setting of this paper?

	\whamb Several references in the OFO literature consider explicit inequality constraints in the optimization problem and corresponding primal-dual methods. We conjecture that the analysis in this paper can be easily extended to address expectation inequality constraints. Could we extend this framework even further and allow for other types of constraints?

	\clearpage

\bibliographystyle{abbrv}
	\bibliography{OFOstuff} 
\appendix

\clearpage
\section{Technical Proofs}

We begin this section by establishing \Cref{t:facts_f}.
	
	\smallskip
	\whamit{Proof of \Cref{t:facts_f}~(i)}
	For each $n$, denote $x^*_n = \comply(u^*_n,\Phi_{n+1})$ and $y^*_n = h^{(n)}(x^*_n)$.
	By the chain rule we obtain,
	\begin{equation}
		\begin{aligned}
			\| \nabla_u \tilf^{(n)}_y(u_n,\Phi_{n+1})  
			\| 
			&= 
			\| A^\transpose_{n+1} C^\transpose [\nabla \tilg^{(n)}_y(y_n)     \| 
            \leq
			\| [C A_{n+1}]^\transpose \| \| \nabla \tilg^{(n)}_y(y_n)   \|
		\end{aligned}
		\label{e:Lipsch_step}
	\end{equation}
	in which $\nabla \tilg^{(n)}_y(y_n) \eqdef \nabla g^{(n)}_y(y_n)-\nabla g^{(n)}_y(y^*_n)$.
	The Lipschitz continuity conditions in (A1) imply that $\| \nabla \tilg^{(n)}_y(y_n)\| \leq  L_g \| C A_{n+1}  \|   \|  \tilu_n   \|$, yielding
	$
	\| \nabla_u \tilf^{(n)}_y(u_n,\Phi_{n+1})  
			\| 
	\leq 
	 \bdd{t:facts_f} \| A_{n+1}  \|^2   \|  \tilu_n   \|
	$, where $\bdd{t:facts_f}$ is a constant. Applying similar steps as the ones outlined above to $\nabla_u f_x$ yields an analogous bound:
    $
	\| \nabla_u \tilf^{(n)}_x(u_n,\Phi_{n+1})  
			\| 
	\leq 
	 L_g \| A_{n+1}  \|^2   \|  \tilu_n   \|
	$.

    Now, by the triangle inequality we have
    \[
    \| \nabla_u \tilf^{(n)}(u_n,\Phi_{n+1}) 
			\|
            \leq (\bdd{t:facts_f} +L_g) \| A_{n+1}  \|^2   \|  \tilu_n \|
    \]
    Squaring both sides of the above equation and taking conditional expectations completes the proof with 
    $L_f =(\bdd{t:facts_f} + L_g) \sigma_\Delta^4 $ in view of the moment bounds in (A3).
	\hfill
	\qed

	\smallskip
	\whamit{Proof of \Cref{t:facts_f}~(ii)}
	Let $x^u_n = \comply(u,\Phi_{n+1})$, $y^u_n = h^{(n)}(x^u_n)$, $x^{u'}_n = \comply(u',\Phi_{n+1})$ and $y^{u'}_n = h^{(n)}(x^{u'}_n)$. Moreover, let $\varrho^t(u,u') = t u + (1-t)u'$.
	Since $g^{(n)}_x$ and $g^{(n)}_y$ are assumed to be strongly convex, we have that for each $n,t$ and any $u,u' \in \clU^{(n)}$ that are $\clF_n$-measurable,
	\begin{equation}
	\begin{aligned}
	f^{(n)}(\varrho^t(u,u'),\Phi_{n+1}) &= g^{(n)}_x(\varrho^t(x^{u}_n,x^{u'}_n)) + g^{(n)}_y(\varrho^t(y^{u}_n,y^{u'}_n))
				\\
				&\leq 
				\varrho^t(f^{(n)}(u,\Phi_{n+1}), f^{(n)}(u',\Phi_{n+1}))  
                \\
                & \quad - \frac{1}{2} t(1-t) \mu [\|  C A_{n+1} (u-u')  \|^2 + \|  A_{n+1} (u-u')  \|^2] 
	\end{aligned}
\label{e:stong_cvx}
	\end{equation}
	
	Now, since $u,u'$ are $\clF_n$-measurable and $\{\Phi_n\}$ is assumed i.i.d., we apply Jensen's inequality to obtain the lower bound:
	\[
	\|   C \barA  (u-u')  \|^2 \leq \Expect[ \|  C A_{n+1}  (u-u') \|^2  | \clF_n ]
	\]
    Similarly, we have $	\|   \barA  (u-u')  \|^2 \leq \Expect[ \|  A_{n+1}  (u-u') \|^2  | \clF_n ]$.
    
    Under (A5) it follows that the eigenvalues of the matrices $\barA^\transpose C^\transpose C \barA$ and $\barA^\transpose  \barA$ are positive. Taking conditional expectations of both sides of \eqref{e:stong_cvx} and using the above lower bound, we obtain strong convexity of $f$ in conditional mean,  which implies part (ii) of \Cref{t:facts_f} with
	\[
	\barmu_f =  \mu [\lambda_{\min} ( \barA^\transpose C^\transpose C\barA)  + \lambda_{\min} ( \barA^\transpose 
 \barA  )]
	\]
	\hfill
	\qed

	\smallskip
	\whamit{Proof of \Cref{t:facts_f}~(iii)}
	Let $x^u_n = \comply(u,\Phi_{n+1})$ and  $y^u_n  = h(x^u_n,r_n)$.
	Using the triangle inequality, the  assumed Lipschitz continuity of $\nabla g^{(n)}_y$ and $\nabla g^{(n)}_x$ in (A1) and H\"{o}lder's inequality, we have the following for a constant $b^\bullet$,
	\[\begin{aligned}
		\|  \nabla_u f^{(n)}(u,\Phi_{n+1}) \|_{2,n}
		&
		\leq 
		\| [C A_{n+1}]^\transpose \nabla g^{(n)}_y(y^u_n) \|_{2,n} 
		 + 
		\|A_{n+1}^\transpose \nabla g^{(n)}_x(x^u_n) \|_{2,n}
		\\
		&
		\leq \| [C A_{n+1}]^\transpose\|_{4,n} \{ b^\bullet
		+ L_g\|  y^u_n\|_{4,n} \}
		 	+ \| [ A_{n+1}^\transpose   \|_{4,n} \{  b^\bullet
		+ L_g\| C x^u_n\|_{4,n}  \}
	\end{aligned}
	\]
	In view of the assumed bound $\| A_{n+1} \|_{4,n} \leq \sigma_\Delta$ in (A3) and the fact that $u^*_n$ is $\clF_n$-measurable, it follows that there exists a potentially larger  constant $b^\circ$ such that $\| x^u_n \|_{4,n}\leq b^\circ$ and $\| y^u_n \|_{4,n}\leq b^\circ$.
    Together with the above equation, this also leads to: $\|  \nabla_u f^{(n)}(u,\Phi_{n+1}) \|_{2,n} \leq \bdd{t:facts_f}$, in which $\bdd{t:facts_f}$ is a constant.
    
    Applications of Jensen's inequality and the triangle inequality complete the proof for $\sigma_f = 2 \bdd{t:facts_f}$. \hfill \qed

		\smallskip

	The proof of \Cref{t:online_tracking_result_limsup} is largely based upon obtaining a contraction for the MSE $\| u_n  -u^*_n \|_2^2$. The next result follows from parts (i) and (ii) of \Cref{t:facts_f} and is important in defining a contractive factor for the MSE.
	\begin{corollary}
		\label[corollary]{t:Lip+Monotone}
		Under the assumptions of \Cref{t:facts_f}~(ii), the following holds:
		\[
		\Expect[\| \tilu_{n-1} -\alpha  \nabla_u \tilf^{(n-1)}(u_{n-1},\Phi_{n})  \|^2 \mid \clF_{n-1}]
		\leq
		\Upsilon_\alpha \| \tilu_{n-1} \|^2
		\]
		where $\Upsilon_\alpha = 1-2\alpha \barmu_f + \alpha^2 L_{f}^2$.\hfill
		\qed
	\end{corollary}

		\smallskip
	
	Next, we obtain an uniform bound on the error resulting from estimating the sequence of true gradients $\{\nabla_u f^{(n)}\}$ via 
    the measurement-constructed sequence 
    $\{\nabla \Obj^{(n)}\}$.

	\begin{lemma}
		\label[lemma]{t:estimation_haclL}
		Suppose that (A1)--(A4) hold. Then for a constant $\bdd{t:estimation_haclL}$,
		\[
		\|  \nabla \haObj^{(n)} -  \nabla_u f^{(n)}(u_n,\Phi_{n+1})  \|_{2,n} 
		\leq 
		\bdd{t:estimation_haclL} \epsy_m
		\]
	\end{lemma}
	\begin{proof}
		From the triangle inequality, we have
		\[ 
		\|  \nabla \haObj^{(n)} -  \nabla_u f^{(n)}(u_n,\Phi_{n+1})  \|_{2,n} 
		\leq 
		\clG^a_n + \clG^b_n + \clG^c_n + \clG^d_n
		\]
		in which
		\[
		\begin{aligned}
			\clG^a_n  &= \| {A_{n+1}^\circ}^\transpose C^\transpose  [\nabla g_y^{(n)}(\hay_n) - \nabla g_y^{(n)}(y_n)] \|_{2,n} 
			\\
			\clG^b_n &= \| [{A_{n+1}^\circ}^\transpose - {A}^\transpose_{n+1}] C^\transpose \nabla g_y^{(n)}(y_n) \|_{2,n} 
			\\
			\clG^c_n &=  \|  {A_{n+1}^\circ}^\transpose [\nabla g_x^{(n)}(\hax_n) - \nabla g_x^{(n)}(x_n)] \|_{2,n} 
			\\
			\clG^d_n &=  \|  [{A_{n+1}^\circ}^\transpose - {A}^\transpose_{n+1}]   \nabla g_x^{(n)}(x_n) \|_{2,n} 
		\end{aligned}
		\]
		We proceed to bound each term. Assumptions (A3), (A4) and the triangle inequality imply the upper bound:
		$
		\|A_{n+1}^\circ \|_{4,n}  \leq \epsy_m + \sigma_\Delta
		$.
		Together with the assumed Lipschitz continuity of $\nabla g^{(n)}_y$ in (A1), the bounds in (A4) and H\"{o}lder's inequality, we have the following: for a constant $b$,
		\[
		\begin{aligned}
			\clG^a_n  & \leq  \| {A_{n+1}^\circ}^\transpose C^\transpose \|_{4,n} \| \hay_n - y_n \|_{4,n}
			\leq   b (\epsy_m + \sigma_\Delta) \epsy_m
			\\
			\clG^b_n & \leq \| [{A_{n+1}^\circ}^\transpose - {A}^\transpose_{n+1} \|_{4,n}  \|\nabla g^{(n)}_y(y_n) \|_{4,n}
			\\
			&\leq \epsy_m  ( \| \nabla g^{(n)}_y(0) \|_{4,n} + L_g\| C x_n \|_{4,n}  + L_g \| D r_n \|_{4,n}) 
       \leq b \epsy_m
		\end{aligned}
		\]
		in which the uniform bound on $\| \nabla g^{(n)}_y(y_n)\|_{4,n}$ follows from similar steps as in the proof of \Cref{t:facts_f}~(iii). 
        
        Repeating the above process for $\nabla g_x^{(x)}$ yields analogous bounds: $ \clG^c_n  \leq b(\epsy_m + \sigma_\Delta) \epsy_m$, $\clG^d_n  \leq b \epsy_m$. This completes the proof with $\bdd{t:estimation_haclL} = 2b(\epsy_m + \sigma_\Delta+1) $.
	\end{proof}
	
		\smallskip
	
		The following shorthand notation will be used throughout the remainder of the appendix: let $\clE_n \eqdef \clE^a_n - \alpha \clE^b_n$ and $\clM_n  \eqdef \clM^a_n +  \clM^b_n$, in which
	\[
	\begin{aligned}
		\clE^a_n &\eqdef \tilu_{n-1} 
		\, , \quad 
		\clE^b_n \eqdef   \nabla_u \tilf^{(n-1)}(u_{n-1},\Phi_{n})  
		\\
		\clM^a_n  &\eqdef \nabla_u f^{(n-1)}(u^*_{n-1},\Phi_{n})  - \nabla_u \Expect[f^{(n-1)}(u^*_{n-1},\Phi_{n})] 
		\\
		\clM^b_n &\eqdef \nabla \haObj^{(n)} -  \nabla_u f^{(n-1)}(u_{n-1},\Phi_{n}) 
	\end{aligned}
	\]
	Moreover, let  $\clD_n  \eqdef ( u^*_{n-1} - u_n^*)^\transpose  (u_n - u_{n-1}^*)$.	
	
	We proceed to bounding the terms $E[\clM_n^\transpose \clE_n  \mid \clF_{n-1}]$ and $E[\clD_n \mid \clF_{n-1}]$ in the next two lemmas. Together with the identity in \Cref{t:Lip+Monotone}, bounds on the first term are crucial in establishing \Cref{t:one_step_couple_haobj}, while bounds on the latter are used to obtain \Cref{t:tracking_result_online}.

	\begin{lemma}
		\label[lemma]{t:Dom_conv}
		Under (A1) and (A3), it follows that 
        \begin{equation}
        \clM^a_n = \nabla_u f^{(n-1)}(u^*_{n-1},\Phi_{n})  -  \Expect[\nabla_u f^{(n-1)}(u^*_{n-1},\Phi_{n})] 
        \label{e:dom_conv}
        \end{equation}
        Consequently, $\Expect[{\clM^a_n}^\transpose \clE^a_n   \mid \clF_{n-1}]=0$.
	\end{lemma}
	\begin{proof}
        The result in \eqref{e:dom_conv} follows directly from the dominated convergence theorem (i.e., the order of expectation and differentiation can be exchanged in the second term of $\clM^a_n$). To prove the remaining identity, we use the facts that $ \clE^a_n$ is $\clF_n$-measurable and $\{\Phi_n\}$ is i.i.d.. Then, we have that 
        \[
        \Expect[\nabla_u {f^{(n-1)}(u^*_{n-1},\Phi_{n})}^\transpose \clE^a_n  \mid \clF_{n-1}] 
		=   \Expect[\nabla_u f^{(n-1)}(u^*_{n-1},\Phi_{n}) ]^\transpose   \clE^a_n
        \]
	\end{proof}

	\begin{lemma}
		\label[lemma]{t:Remaining_Bounds}
		Under (A1)--(A4), the following bounds hold:
		\whamrm{(i)} 
		$\displaystyle
		| \Expect[ {\clM^a_n}^\transpose 
		\clE^b_n \mid \clF_{n-1}  ]|
		\leq 
		2 \sigma_f L_f b_\clU
		$
		\whamrm{(ii)} 
		$\displaystyle
		| \Expect[ {\clM^b_n}^\transpose 
		\clE^a_n \mid \clF_{n-1}  ]|
		\leq 
		2 \bdd{t:estimation_haclL} \epsy_m b_\clU
		$
		\whamrm{(iii)} 
		$\displaystyle
		| \Expect[ {\clM^b_n}^\transpose 
		\clE^b_n \mid \clF_{n-1}  ]|
		\leq 
		2 \bdd{t:estimation_haclL} \epsy_m L_f b_\clU
		$
	\end{lemma}
	\begin{proof}
        The Cauchy-Schwarz inequality yields
		\[
		\begin{aligned}	
			|\Expect[ {\clM^a_n}^\transpose  \clE^b_n  \mid \clF_{n-1}]|
			&\leq
			\|  \clM^a_n\|_{2,n-1} 
			\|  \clE^b_n\|_{2,n-1}
			\\
			|\Expect[{\clM^b_n}^\transpose \clE^a_n \mid \clF_{n-1}]|
			&\leq  \|  \clM^b_n\|_{2,n-1}
			\|  \clE^a_n\|_{2,n-1}
			\\
			|\Expect[{\clM^b_n}^\transpose  \clE^b_n  \mid \clF_{n-1}]|
			&\leq
			\|  \clM^b_n\|_{2,n-1} 
			\|  \clE^b_n\|_{2,n-1}
		\end{aligned}
		\]
		Under (A2), we have that  $\|  \clE^a_n\|_{2,n-1} \leq 2 b_\clU$, which also implies the following, via \Cref{t:facts_f}~(i): $\|  \clE^b_n\|_{2,n-1} \leq 2 L_fb_\clU$. Moreover, applications of \eqref{e:dom_conv}, \Cref{t:facts_f}~(iii) and \Cref{t:estimation_haclL} lead to the bounds: $\|  \clM^a_n\|_{2,n-1} \leq \sigma_f$ and $\|  \clM^b_n\|_{2,n-1} \leq \bdd{t:estimation_haclL} \epsy_m$, completing the proof.
	\end{proof}

	\begin{lemma}
		\label[lemma]{t:one_step_couple_haobj}
		Under  the assumptions of \Cref{t:tracking_result_online}, the following holds
		\[
		\Expect[\|  u_n - u^*_{n-1} \|^2  \mid \clF_{n-1} ]
		\leq 
		\Upsilon_\alpha \| \tilu_{n-1} \|^2
		+  q_\alpha
		\]
		with $q_\alpha= \bdd{t:one_step_couple_haobj}  [ \alpha^2 (2 \sigma_f^2  + \sigma_f+ 2 \epsy^2_m + \epsy_m) + \alpha\epsy_m ]$, in which 
        $\bdd{t:one_step_couple_haobj}$ is a constant depending upon $b_\clU$ and $L_f$.
	\end{lemma}
	\begin{proof}
        Let $\upbeta_{n-1} = u^*_{n-1} - \alpha \nabla_u \Expect[f^{(n-1)}(u^*_{n-1},\Phi_n) ]$. 
        By the fact that $u^*_{n-1}$ satisfies a fixed point equation, we have that
		\[
		\begin{aligned}
			\|  u_n - u^*_{n-1}   \|^2
			&= \| \Proj_{\clU^{(n)}}\{u_{n-1}  - \alpha \nabla \haObj^{(n-1)} \}  - \Proj_{\clU^{(n)}}\{ \upbeta_{n-1}  \}   \|^2
			\\
			&\leq 
			\| u_{n-1} - \upbeta_{n-1}  - \alpha \nabla \haObj^{(n-1)}  \|^2
			\\
			& =
			\| \clE_n \|^2+ \alpha^2 \|\clM_n \|^2  - 2\alpha {\clM^a_n}^\transpose \clE^a_n +  \clH_n
		\end{aligned}
		\]
		in which the second inequality follows from the non-expansiveness property of the projection operator and  $\clH_n = 2\alpha^2  [{\clM^a_n}^\transpose \clE^b_n  + {\clM^b_n}^\transpose \clE^b_n  ]          - 2 \alpha{\clM^b_n}^\transpose \clE^a_n   $.

		Upon taking conditional expectations of both sides and applying \Cref{t:Lip+Monotone} and \Cref{t:Dom_conv}, we obtain the upper bound
		\[
		\begin{aligned}
		\|  u_n - u^*_{n-1}   \|^2_{2,n-1} &\leq \Upsilon_\alpha \|  \tilu_{n-1} \|^2 + \alpha^2 \| \clM_n \|^2_{2,n-1} 
		 + |\Expect[\clH_n |\clF_{n-1} ] |
		\end{aligned}
		\]
		It remains to bound the last two terms in the right hand side. Using the triangle inequality and the bounds in \Cref{t:Remaining_Bounds}, we have
		\[
		|\Expect[\clH_n |\clF_{n-1} ] | \leq  \bdd{t:one_step_couple_haobj} \{\alpha^2 [ \sigma_f + \epsy_m] + \alpha \epsy_m\}
		\]
		in which $\bdd{t:one_step_couple_haobj}$ is a constant depending upon $b_\clU$ and $L_f$.
		The remaining term is bounded similarly: applying \eqref{e:dom_conv}, \Cref{t:facts_f}~(iii), \Cref{t:estimation_haclL} and the triangle inequality we have $\| \clM_n \|_{2,n-1}^2 \leq 2[\sigma^2_f +  (\bdd{t:estimation_haclL} \epsy_m)^2]$,
		 which completes the proof.
	\end{proof}

	\begin{lemma}
		\label[lemma]{t:middle_track}
		Under the assumptions of \Cref{t:tracking_result_online}, the following bound holds:
		\[
		|\Expect[\clD_n \mid \clF_{n-1}]| 
		\leq
		\beta_{n-1}
		\]
		in which 
		\[
		\beta_{n-1} = \psi_{n-1} \sum_{i=0}^{n-1} \Upupsilon^{i/2} \sqrt{q_\alpha}
        +
        \sqrt{\Upupsilon_\alpha} \psi_{n-1} \sum_{i=0}^{n-2} \Upupsilon_\alpha^{i/2} \psi_{n-i-2}
		+  \psi_{n-1}\Upupsilon^{n/2}\| \tilu_0 \|
		\]
		where $q_\alpha$ is given by  \Cref{t:one_step_couple_haobj}.
	\end{lemma}
	\begin{proof}
		The Cauchy-Schwarz inequality yields
		\begin{equation}
			\begin{aligned}
				|\Expect[\clD_n \mid \clF_{n-1}]| 
				&\leq
				\psi_{n-1}  \| u_n - u_{n-1}^*\|_{2,n-1}
				\\
				&\leq 
				\psi_{n-1} [\sqrt{\Upupsilon_\alpha} \| \tilu_{n-1} \|  + \sqrt{q_\alpha}]
                 \\
				&\leq 
				\psi_{n-1} [\sqrt{\Upupsilon_\alpha} \| u_{n-1} - u_{n-2}^*\|_{2,n-1} + \sqrt{\Upupsilon_\alpha} \psi_{n-2}  + \sqrt{q_\alpha}]
			\end{aligned}
			\label{e:step1_mid}
		\end{equation}
		where the last inequality was obtained from an application of \Cref{t:one_step_couple_haobj} and the triangle inequality.
		Repeating this process recursively yields the desired result.
	\end{proof}
	
		\smallskip
	Equipped with \Cref{t:one_step_couple_haobj} and \Cref{t:middle_track}, we are ready to establish the finite-time bounds in \Cref{t:tracking_result_online}.

	\smallskip
	\whamit{Proof of \Cref{t:tracking_result_online}}
	Expanding the square in $\|\tilu_n -u^*_{n-1} + u^*_{n-1}\|^2_{2,n-1}$ yields
	\[
	\begin{aligned}
		\| \tilu_n  \|^2_{2,n-1}
		&\leq 
		\psi_{n-1}^2 + 
		\| u_n - u_{n-1}^* \|^2_{2,n-1} + 2 |\Expect[\clD_n \mid \clF_{n-1}]| 
		\\
		&\leq 
		\psi^2_{n-1}
		+
		\Upsilon_\alpha \| \tilu_{n-1} \|^2 +  q_\alpha + 2\beta_{n-1}
	\end{aligned}
	\]
	where the last bound follows from applications of \Cref{t:one_step_couple_haobj} and \Cref{t:middle_track}.
	Taking expectations of both sides and repeating this process recursively yields the final result. \hfill \qed

		\smallskip
	Finally, \Cref{t:online_tracking_result_limsup} follows as a corollary to \Cref{t:tracking_result_online}.
	\smallskip
	\whamit{Proof of \Cref{t:online_tracking_result_limsup}}
	Upon choosing $\alpha < \frac{\barmu_f}{ L^2_f}$, it follows that $\Upsilon_\alpha \leq \sqrt{\Upsilon_\alpha} <1$ and $\Upupsilon_\alpha < 1-\barmu_f \alpha$. Thus, we obtain the following, via the geometric series formula: for a fixed $N$ and each $k$, 
			\[
	\beta_{N-k-1} \leq  \bargamma  \sqrt{q_\alpha}   \frac{[1 -\Upupsilon^{(N-k-1)/2}_\alpha ]}{1-\sqrt{\Upupsilon_\alpha}}
+
\bargamma^2 \sqrt{\Upupsilon_\alpha}     \frac{[1 -\Upupsilon^{(N-k-2)/2}_\alpha ]}{1-\sqrt{\Upupsilon_\alpha}}
	+  \bargamma \Upupsilon^{(N-k)/2}_\alpha\| \tilu_0 \|
	\]
	Substituting the above identity into \eqref{e:track_finitetime_online}, using the geometric series formula once more and taking the limit supremum of both sides yields
	\[
	\begin{aligned}
		\limsup_{N \to \infty} 
		\| \tilu_{N}  \|_2^2
		&\leq  \frac{1}{\barmu_f \alpha} [q_\alpha + \bargamma^2 ] + [\bargamma \sqrt{q_\alpha} +  \bargamma^2 \sqrt{\Upupsilon_\alpha} ] \frac{2}{1-\sqrt{\Upupsilon_\alpha}}\frac{1}{\barmu_f \alpha}
        \\
        &
		\leq  \frac{1}{\barmu_f \alpha} [q_\alpha + \bargamma^2 ] +  2\frac{\bargamma \sqrt{q_\alpha}+ \bargamma^2 \sqrt{\Upupsilon_\alpha}}{\barmu^2_f \alpha^2} [1+\sqrt{\Upupsilon_\alpha}]
	\end{aligned}		
	\]
	The proof is complete upon using the fact that $\sqrt{\Upupsilon_\alpha}<1$ to bound the last term and substituting $q_{\alpha}$ from \Cref{t:tracking_result_online}, 
	\[
	2\frac{\bargamma \sqrt{q_\alpha}+ \bargamma^2 \sqrt{\Upupsilon_\alpha}}{\barmu^2_f \alpha^2} [1+\sqrt{\Upupsilon_\alpha}]
	\leq 
	\frac{4 \bargamma}{\barmu^2_f \alpha^{3/2}}  
	\sqrt{\bdd{t:one_step_couple_haobj} [ \alpha (2 \sigma_f^2  + \sigma_f+ 2 \epsy^2_m + \epsy_m) + \epsy_m ]}
    +
    \frac{4}{\barmu^2_f \alpha^{2}}  \bargamma^2
 	\] 
	\hfill
	\qed

\end{document}